\documentclass[secthm,seceqn,amsthm,ussrhead,12pt]{amsart}
\usepackage[utf8]{inputenc}
\usepackage[english]{babel}

\usepackage{amssymb,amsmath,amsthm,amsfonts,xcolor,enumerate,hyperref,comment,longtable,cleveref}
\usepackage{soul}
\usepackage{times}
\usepackage{cite}
\usepackage{pdflscape}
\usepackage{ulem}
\usepackage[mathcal]{euscript}
\usepackage{tikz}
\usepackage{hyperref}
\usepackage{cancel}
\usepackage{stmaryrd}

\usetikzlibrary{arrows}

\newtheorem{theorem}{Theorem}

\newtheorem{proposition}[theorem]{Proposition}

\newtheorem{definition}[theorem]{Definition}

\newtheorem{problem}[theorem]{Problem}

\usepackage{stmaryrd}
\usepackage{xcolor}

\newcommand{\A}{\mathbf{A}}

\begin{document}

\noindent{\Large   
Local and 2-local derivations of   
simple $n$-ary algebras}

    \medskip

   {\bf
    Bruno Leonardo Macedo Ferreira$^{a}$,
    Ivan Kaygorodov$^{b}$ \& Karimbergen Kudaybergenov$^{c}$  \\

    \medskip
 
}

{\tiny

$^{a}$ Federal University of Technology, Guarapuava, Brazil

$^{b}$ CMCC, Universidade Federal do ABC, Santo Andr\'e, Brazil

$^{c}$  V.I.Romanovskiy Institute of Mathematics Uzbekistan Academy of Sciences,  Tashkent,   Uzbekistan

\

   E-mail addresses:

\smallskip
   
Bruno Leonardo Macedo Ferreira   
   (brunoferreira@utfpr.edu.br)

    Ivan Kaygorodov (kaygorodov.ivan@gmail.com)

    Karimbergen Kudaybergenov (karim2006@mail.ru)

}

\ 


\

AMS: 17A30, 17A42.

\ 

{\bf Abstract}: 
In the present paper we prove that every local and $2$-local
derivation of the complex finite dimensional  simple Filippov algebra 
 is a derivation.
As a corollary we have the description of all local and $2$-local derivations of complex finite dimensional semisimple Filippov  algebras.
All local derivations of the ternary Malcev algebra $M_8$ are described.  
It is the first example of a   finite-dimensional simple algebra which admits pure  local  derivations, i.e. algebra admits a local  derivation which is not a derivation.

\section*{Introduction}
The study of local derivations
started with Kadison’s article
 \cite{k90}.  A similar notion, which characterizes non-linear generalizations of derivations, 
was introduced by  \v{S}emrl as  $2$-local derivations.
In his paper \cite{Semrl97} was 
proved that a $2$-local derivation of the algebra
$B(H)$ of all bounded linear operators on the infinite-dimensional
separable Hilbert space $H$ is a derivation. 
After these works, appear numerous new results related to the description of local and $2$-local derivations of associative algebras (see, for example, \cite{Khrypchenko19}).
The study of local and $2$-local derivations of non-associative algebras was initiated in some 
papers of  Ayupov and  Kudaybergenov (for the case of Lie algebras, see \cite{ak17,ak16}).
In particular, they proved that there are no pure local and $2$-local derivations on semisimple finite-dimensional Lie algebras.
In \cite{AyuKudRak16} it is also given examples of $2$-local derivations on nilpotent Lie algebras which are not derivations.
After the cited works, the study of local and $2$-local derivations was continued for Leibniz algebras \cite{akoz} and Jordan algebras \cite{aa17}.
Local automorphisms and $2$-local automorphisms, also were studied in many cases,
for example, they were studied on Lie algebras \cite{ak17,c19}.
The present paper is devoted to the study local and $2$-local derivations of $n$-ary simple algebras, 
such that Filippov algebras and the ternary Malcev algebra $M_8.$
Early, some certain types of generalized derivations of these algebras were described in \cite{Ivan,k14}.

Our brief
introduction finishes with two principal definitions.

\begin{definition}
Let $\A$ be an $n$-ary algebra. 
A linear map $\nabla : \A \rightarrow \A$ is called a local derivation, if for
any element $x \in \A$ there exists a derivation ${\mathfrak D}_x: \A \rightarrow \A$ such that $\nabla(x) = {\mathfrak D}_x(x)$.
\end{definition}
\begin{definition}A (not necessary linear) map $\Delta: \A \rightarrow \A$ is called a $2$-local derivation, if for
any two elements $x, y \in \A$ there exists an derivation ${\mathfrak D}_{x,y} : \A \rightarrow \A$ such that
$\Delta(x) = {\mathfrak D}_{x,y}(x)$, $\Delta(y) = {\mathfrak D}_{x,y}(y)$.
\end{definition}

\section{Local and $2$-local derivations of Filippov algebras}
\subsection{Preliminaries}
A Filippov algebra, whose definition appeared in \cite{filippov}, is defined as an algebra $L$ with one anticommutative
$n$-ary operation $[x_1,\ldots ,x_n]$ satisfying the identity
$$[[x_1,\ldots ,x_n], y_2,\ldots ,y_n] = \sum^n_{i=1} [x_1,\ldots , [x_i, y_2,\ldots ,y_n],\ldots ,x_n].$$
An example of an $(n + 1)$-dimensional $n$-ary Filippov algebra is the algebra with the basis ${\ell} = \{e_1,\ldots ,e_{n+1}\}$
and the multiplication table
\begin{center}
$[e_1,\ldots , \hat{e}_i, \ldots ,e_{n+1}]=(-1)^{n+i+1}e_i,$
\end{center}
where $\hat{e}_i$ denotes the omission of the element $e_i$ from the $n$-ary product. We denote this algebra by $A_{n+1}.$
We will consider $(n-1)$-ary $n$-dimensional algebras $A_n$ for $n \geq 4.$
As mentioned in \cite{Ling}, the algebras of type $A_{n}$ exhaust all simple finite-dimensional $(n-1)$-ary Filippov
algebras over an algebraically closed field of characteristic zero.
Thanks to \cite{Ivan} we have the following description of the matrix of a derivation of the $n$-ary algebra $A_{n}$. 
\begin{proposition}\label{chave}
A linear map ${\mathfrak D} : A_{n} \rightarrow A_{n}$ is a derivation of the $(n-1)$-ary algebra $A_{n}$ if and only if the matrix
of ${\mathfrak D}$ has the following matrix form:

\begin{displaymath}
[{\mathfrak D}]_{\ell} = \begin{pmatrix}
0&x_{12}&x_{13}& \ldots & x_{1k}& \ldots&x_{1n-1} & x_{1n}\\
-x_{12}&0&x_{23}& \ldots & x_{2k} & \ldots&x_{2n-1} &x_{2n}\\
-x_{13}&-x_{23}&0& \ldots & x_{3k} & \ldots& x_{3n-1}&x_{3n}\\
\vdots & \vdots & \vdots & \ddots & \vdots & \ddots & \vdots& \vdots\\
-x_{1k} & -x_{2k} & -x_{3k} &\ldots & 0 & \ldots &x_{kn-1}& x_{kn}\\
\vdots & \vdots & \vdots & \ddots & \vdots & \ddots & \vdots & \vdots\\
-x_{1n-1}&-x_{2n-1}&-x_{3n-1}& \ldots & -x_{kn-1}& \ldots &0 & x_{n-1n}\\
-x_{1n}&-x_{2n}&-x_{3n}& \ldots & -x_{kn}& \ldots & -x_{n-1n} & 0\\
\end{pmatrix},
\end{displaymath}
that is, $x_{ii} = 0$ and $x_{ij} + x_{ji} = 0$ to $i \neq j$.
\end{proposition}

\subsection{Local derivations of semisimple Filippov algebras}
In the present subsection, we proved that each local derivation of the complex finite-dimensional simple Filippov $n$-ary $(n>2)$ algebra is a derivation. 
As a corollary, jointed with the results from \cite{ak16, Ling}, we have the same statement for all complex finite-dimensional semisimple Filippov $n$-ary $(n>1)$ algebras.
\begin{theorem}\label{local}
Each local derivation of $A_{n}$ is a derivation.
\end{theorem}

\begin{proof}
Let $\nabla$ be an arbitrary local derivation of $A_{n}$, by definition we have
$$\nabla(x) = D_x(x).$$
Let us consider $B$ and $A_x$ the matrix of the linear operator $\nabla$ and $D_x$ respectively. Thus $B(x) = A_x(x)$,
$$B(x) = \begin{pmatrix}
b_{11}&b_{12}&b_{13}& \ldots & b_{1k}& \ldots&b_{1n-1} & b_{1n}\\
b_{21}&b_{22}&b_{23}& \ldots & b_{2k} & \ldots&b_{2n-1} &b_{2n}\\
b_{31}&b_{32}&b_{33}& \ldots & b_{3k} & \ldots& b_{3n-1}&b_{3n}\\
\vdots & \vdots & \vdots & \ddots & \vdots & \ddots & \vdots\\
b_{k1} & b_{k2} & b_{k3} &\ldots & b_{kk} & \ldots &b_{kn-1}& b_{kn}\\
\vdots & \vdots & \vdots & \ddots & \vdots & \ddots & \vdots\\
b_{n-11}&b_{n-12}&b_{n-13}& \ldots & b_{n-1k}& \ldots &b_{n-1n-1} & b_{n-1n}\\
b_{n1}&b_{n2}&b_{3n}& \ldots & b_{nk}& \ldots & b_{nn-1} & b_{nn}\\
\end{pmatrix}\begin{pmatrix} x_1\\x_2\\x_3\\ \vdots \\ x_k\\ \vdots \\ x_{n-1} \\ x_n  \end{pmatrix} =$$
$$\begin{pmatrix}
0&a^x_{12}&a^x_{13}& \ldots & a^x_{1k}& \ldots&a^x_{1n-1} & a^x_{1n}\\
-a^x_{12}&0&a^x_{23}& \ldots & a^x_{2k} & \ldots&a^x_{2n-1} &a^x_{2n}\\
-a^x_{13}&-a^x_{23}&0& \ldots & a^x_{3k} & \ldots& a^x_{3n-1}&a^x_{3n}\\
\vdots & \vdots & \vdots & \ddots & \vdots & \ddots & \vdots\\
-a^x_{1k} & -a^x_{2k} & -a^x_{3k} &\ldots & 0 & \ldots &a^x_{kn-1}& a^x_{kn}\\
\vdots & \vdots & \vdots & \ddots & \vdots & \ddots & \vdots\\
-a^x_{1n-1}&-a^x_{2n-1}&-a^x_{3n-1}& \ldots & -a^x_{kn-1}& \ldots &0 & a^x_{n-1n}\\
-a^x_{1n}&-a^x_{2n}&-a^x_{3n}& \ldots & -a^x_{kn}& \ldots & -a^x_{n-1n} & 0\\
\end{pmatrix}\begin{pmatrix} x_1\\x_2\\x_3\\ \vdots \\ x_k\\ \vdots \\ x_{n-1} \\ x_n  \end{pmatrix} = A_x(x),$$
for all $x \in A_{n+1}$. 
Taking $\Xi_k = (0, \cdots, 0 , \underbrace{1}_{k}, 0 , \cdots, 0)^T$ for all $k= 1, \ldots  , n$ we get $b_{kk} = 0$ for $k = 1, \ldots , n$. 
Hence we obtain 
\begin{center}
    $B(\Xi_k+\Xi_l)=A_{\Xi_k+\Xi_l}(\Xi_k+\Xi_l).$
\end{center}
Hence, $A_{\Xi_k+\Xi_l}$ is an antisymmetric matrix, 
it gives that $b_{kl}=-b_{lk}$ and the matrix $B$ is antisymmetric.
The last gives that $\nabla$ is a derivation.


\end{proof}

\subsection{$2$-Local derivations of semisimple Filippov algebras}
In the present subsection, we prove that each local derivation of the complex finite-dimensional simple Filippov $n$-ary $(n>2)$ algebra is a derivation. 
As a corollary, jointed with the results from \cite{AyuKudRak16, Ling}, we have the same statement for all complex finite-dimensional semisimple Filippov $n$-ary $(n>1)$ algebras.
\begin{theorem}\label{2local}
Each $2$-local derivation of $A_{n}$ is a derivation.
\end{theorem}

\begin{proof}

Let $\Delta$ be an arbitrary $2$-local derivation of $A_{n}$. Then, by the definition, for
every element $a, b \in A_{n}$, there exists a derivation $D_{a,b}$ of $A_{n}$ such that
$$\Delta(a) = D_{a,b}(a), \ \ \ \Delta(b) =  D_{a, b} (b).$$
By Proposition \ref{chave}, the matrix $A^{a,b}$ of the derivation $D_{a,b}$ has the following matrix
form:
$$
A^{a,b} = \begin{pmatrix}
0&x^{a,b}_{12}&x^{a,b}_{13}& \ldots & x^{a,b}_{1k}& \ldots&x^{a,b}_{1n}  \\
-x^{a,b}_{12}&0&x^{a,b}_{23}& \ldots & x^{a,b}_{2k} & \ldots&x^{a,b}_{2n} \\
-x^{a,b}_{13}&-x^{a,b}_{23}&0& \ldots & x^{a,b}_{3k} & \ldots& x^{a,b}_{3n}\\
\vdots & \vdots & \vdots & \ddots & \vdots & \ddots & \vdots   \\
-x^{a,b}_{1k} & -x^{a,b}_{2k} & -x^{a,b}_{3k} &\ldots & 0 & \ldots &x^{a,b}_{kn}\\
\vdots & \vdots & \vdots & \ddots & \vdots & \ddots & \vdots  \\
-x^{a,b}_{1n}&-x^{a,b}_{2n}&-x^{a,b}_{3n}& \ldots & -x^{a,b}_{kn}& \ldots &0 \\
\end{pmatrix}
$$
$$
=
 \begin{pmatrix}
0& & T^{a,b} \\   
 &\ddots &  \\
-T^{a,b} & & 0\\
\end{pmatrix}
$$
Let $a = \sum^{n}_{i=1}\lambda_{i} e_i$ be an arbitrary element from $A_{n}$. 
For every $v \in A_{n}$ there exists a derivation $D_{v,a}$ such that
$$\Delta(v) = D_{v,a}(v), \ \ \ \Delta(a) = D_{v,a}(a).$$
Then from $$D_{e_{n}, v}(e_{n}) = D_{e_{n}, a}(e_{n}), \ \ \ \ \ \ v \in A_{n},$$
it follows that $T^{e_{n}, v}_{i,n} = T^{e_{n}, a}_{i,n}$
for $i = 1, \ldots, n.$
Then we can write 

$$A^{e_{n}, a} = \begin{pmatrix}
\begin{tabular}{c c c | c}
0& &$\hat{T}$ &  \\
& 0&   & ${\bf T}^{e_{n},v}$\\
 $-\hat{T}$& & $\ddots$ &   \\ \hline
& $-({\bf T}^{e_{n},v})^T$ & & 0\\
\end{tabular} 
\end{pmatrix}$$
where 
$ {\bf T}^{e_{n},v} = 
\begin{pmatrix}
x^{e_{n},v}_{1n}\\
x^{e_{n},v}_{2n}\\
\vdots \\
x^{e_{n},v}_{n-1n}
\end{pmatrix}.
$
Hence, 
\begin{center}
$\Delta(a) = D_{e_{n},a}(a) = \sum_{i=1}^{n-1} \mu_{i}^{e_{n},a}e_i + \sum_{i=1}^{n}(-x^{e_{n},v}_{in}\lambda_{i} )e_{n},$
\end{center}
for some elements $\mu^{e_{n},a}_{i} \in \mathbb{F}$. 
Similarly, taking $e_j$ for each $j = n-1, n-2, \ldots, 2$ we have from
\begin{center}
$D_{e_j, v}(e_j) = D_{e_j, a}(e_j), \ \ \ \ v \in A_{n},$
\end{center}
we have the following ${  T}^{e_j,v}_{i,j} = {  T}^{e_j,a}_{i,j}$ for each $j= n-1, n-2, \ldots, 2$ and $T^{e_j,v}_{i,j} = - T^{e_j,v}_{j,i}$.
Hence, 
$$\Delta(a) = D_{e_i, a}(a), \ \ \text{for each} \  \  i = 1, \ldots, n.$$
Note that 
\begin{center}$\Delta(a) = \sum_{i=1}^{n}\sum_{j=1}^{n}(-T_{i,j}^{e_i,v_j}\lambda_{e_j})e_i, \, \, v_j \in A_{n+1}, \, \,  j= 1, \ldots n.$
\end{center}
Therefore the mapping $\Delta$ is linear and it is a local derivation. By Theorem \ref{local}, we get that $\Delta$ is a derivation. This completes the proof.   
\end{proof}

\section{Local   derivations of the ternary Malcev algebra $M_8$}
\subsection{Preliminaries}
The idea of introducing a generalization of Filippov algebras comes from binary Malcev algebras and it was realized in a paper of Pojidaev \cite{app}.
He defined $n$-ary Malcev algebras, generalizing Malcev algebras and $n$-ary Fillipov algebras.
For construction of the most  important example of $n$-ary Malcev (non-Filippov) we denote by $\A$ a composition algebra  with an involution $\bar{} : a \mapsto \bar{a}$ and unity $1$.
The symmetric, bilinear form $\left\langle x,y\right\rangle = \frac{1}{2} (x\bar{y} + y\bar{x})$ defined on $A$ is assumed to be
nonsingular.
If
$\A$ is equipped with a ternary multiplication $[\cdot,\cdot,\cdot]$ by the rule
\begin{center}$[x, y, z] = (x\bar{y})z - \left\langle y, z\right\rangle x + \left\langle x, z \right\rangle y - \left\langle x, y \right\rangle z,$
\end{center}
then $\A$ becomes a ternary  Malcev algebra \cite{app}, 
which will be denoted by $M(\A)$. 
If $\dim \A = 8$ then $M(\A)$ is  simple ternary Malcev (non-$3$-Lie) algebra and we denote it by $M_8$.

Let $A$ be that above mentioned composition algebra and assume that $1, a,b, c$ are orthonormal elements in $A.$ Choose the following basis of $M_8$ (see \cite{paulo}):
 \begin{longtable}{llll}
$e_1=1,$ & $e_2=a,$ & $e_3=b,$ & $e_4=ab,$\\
$e_5=c,$ & $e_6=ac,$ & $e_7=bc,$ & $e_8=abc.$
\end{longtable}
Further we need to the following properties of basis elements (see \cite{paulo}). For each $i\in \{2, \ldots, 8\},$ it is
possible to choose $j, k, l, m, s, t,$  all depending on $i,$ such that
\begin{align}\label{ets}
e_i = e_1 e_i = e_j e_k = e_l e_m = e_s e_t \,\, \textrm{and}\,\,
e_k e_m = e_t. 
\end{align}

Thanks to \cite{paulo} we have the following description of the basis of the algebra of derivations of  $M_8:$ 
$$\mathcal{B} = \left\{ 
\begin{array}{l}
\Delta_{23} - \Delta_{14}, \Delta_{24} + \Delta_{13}, \Delta_{25} - \Delta_{16},  \Delta_{26} + \Delta_{15},  \Delta_{27} + \Delta_{18}, \Delta_{28} - \Delta_{17}, \Delta_{34} - \Delta_{12},\\
\Delta_{35} - \Delta_{17},  \Delta_{36} - \Delta_{18}, \Delta_{37} + \Delta_{15}, \Delta_{38} + \Delta_{16}, \Delta_{45} - \Delta_{18}, \Delta_{46} + \Delta_{17}, \Delta_{47} - \Delta_{16}, 
 \\ \Delta_{48} + \Delta_{15}, \Delta_{56} - \Delta_{12}, \Delta_{57} - \Delta_{13}, 
  \Delta_{58} - \Delta_{14},
\Delta_{67} + \Delta_{14}, \Delta_{68} - \Delta_{13}, \Delta_{78} + \Delta_{12}\end{array} \right\},$$
where $\Delta_{ij} = e_{ij} - e_{ji}$ and $e_{ij}$ are the ordinary matrix units.
  
\begin{proposition} 
A linear map $\mathfrak{D} : M_{8} \rightarrow M_{8}$ is a derivation of the algebra $M_{8}$ if and only if the antisymmetric matrix
of $D$ has the following matrix form:
\begin{align}\label{derm8}
[\mathfrak{D}]_{\mathcal{B}} = 
\begin{pmatrix}
0        & -\gamma_1 & -\gamma_2 & -\gamma_3 & -\gamma_4 & -\gamma_5 & -\gamma_6 & -\gamma_7 \\
\gamma_1 & 0    & -\alpha_1&-\alpha_2& -\alpha_3&-\alpha_4&-\alpha_5&-\alpha_6  \\
\gamma_2 & \alpha_1 & 0  & -\alpha_7 &-\alpha_8&-\alpha_9&-\alpha_{10}&-\alpha_{11}\\
\gamma_3 & \alpha_2 & \alpha_7 & 0  &-\alpha_{12}&-\alpha_{13}&-\alpha_{14}&-\alpha_{15} \\
\gamma_4 & \alpha_3 & \alpha_8  & \alpha_{12} & 0 &-\alpha_{16}&-\alpha_{17}& -\alpha_{18}\\
\gamma_5 & \alpha_4 & \alpha_9  & \alpha_{13} &   \alpha_{16} &  0 &-\alpha_{19}&-\alpha_{20}\\
\gamma_6 & \alpha_5  & \alpha_{10} &  \alpha_{14} &   \alpha_{17}   &   \alpha_{19} & 0  & -\alpha_{21}\\
\gamma_7& \alpha_6  & \alpha_{11} & \alpha_{15} & \alpha_{18} & \alpha_{20} & \alpha_{21} & 0
\end{pmatrix},
\end{align}
where 
\begin{center}
$\gamma_1=-\alpha_7-\alpha_{16}+\alpha_{21},$ \, 
$\gamma_2=\alpha_2-\alpha_{17}-\alpha_{20},$ \, 
$\gamma_3=-\alpha_1-\alpha_{18}+\alpha_{19},$ \\

$\gamma_4=\alpha_4+\alpha_{10}+\alpha_{15},$ \,  
$\gamma_5=-\alpha_3+\alpha_{11}-\alpha_{14},$ \\

$\gamma_6=-\alpha_6-\alpha_8+\alpha_{13},$ \,  
$\gamma_7=\alpha_5-\alpha_9-\alpha_{12}.$
\end{center}
\end{proposition}

\subsection{Local  derivations of $M_8$}
In the present subsection, we shall give a description of all local derivations of $M_8.$ As the principal result, we have that $M_8$ admits local derivations that are not derivations.
Which gives the first known example of a simple finite-dimensional algebra admitting pure local derivations.
The quotient space of the space of local derivation by the space of derivations of $M_8$ is of dimension $7.$

Recall \cite{paulo} that, if $1, u,v,w \in A$ are orthonormal, then
\begin{align}\label{uvw}
\overline{u}v\overline{u} = -\overline{v},\,\, (u\overline{v})w= -(u\overline{w})v,\,\, u(\overline{v}w)=-v(\overline{u}w). 
\end{align} 
Note that $u^2=v^2=w^2=-1.$ Thus using   Moufang identity $(uv)(wu)  =  u((vw)u)$ for composition algebra and the above first identity  we get that
 \begin{align}\label{mou}
(uv)(wu)=vw. 
\end{align}

 \begin{proposition}\label{aut}
 Let $x, y\in M_8$ be the elements such that $x^2=-1,$  $y\in  \left\{x\right\}^\perp, y^2=-1.$ Then there exists 
 $\Phi\in {\rm Aut}(M_8)$ such that 
 \begin{center}
$\Phi(e_2)=x$ and $\Phi(e_3)=y.$
 \end{center}
 \end{proposition}
 
 \begin{proof}  Since $x^2=-1,$  $y\in  \left\{x\right\}^\perp, y^2=-1,$ it follows that 
$\{e_1, x, y, xy\}$ is an orthonormal system.  Take an element $z\in \{e_1, x, y, xy\}^\perp$ such that $z^2=-1.$
 Using \eqref{uvw} we can infer that 
 $$
 \left\{e_1, x, y, xy, z, xz, yz, xyz\right\}
 $$
is an orthonormal system, in particular, for any three different elements $u, v, w$ from the above system identities from \eqref{uvw} are true.

Define a linear mapping $\Phi$ on $A$ on basis elements as follows:
 \begin{longtable}{llll}
$\Phi(e_1)=e_1,$ & $\Phi(e_2)=x,$ & $\Phi(e_3)=y,$ & $\Phi(e_4)=xy,$\\
$\Phi(e_5)=z,$ & $\Phi(e_6)=xz,$ & $\Phi(e_7)=yz,$ & $\Phi(e_8)=xyz.$
\end{longtable}
Using identities  \eqref{ets}, \eqref{uvw}, \eqref{mou}  we obtain that $\Phi$ is an automorphism of the composition algebra $A.$ Since any automorphsim of $A$  
commutes with the involution and hence, it preserves bilinear form  $\left\langle \cdot,\cdot\right\rangle.$  
It follows that $\Phi$ is also an automorphism of the ternary algebra $M_8.$
 \end{proof}
 
\begin{theorem}\label{localM8}
A linear mapping $\nabla$ on  $M_{8}$ is a local derivation if and only if its matrix  is antisymmetric. 
In particular, the dimension of the space 
$LocDer M_{8}$ of all local derivations of $M_8$ is equal to $28.$
\end{theorem}

\begin{proof} Let $\nabla$ be a local derivation on $M_8.$ By a similar argument as in the proof of Theorem~\ref{local} we obtain that the matrix of $\nabla$ is antisymmetric.

Let $\nabla:M_{8}\to M_{8}$ be an arbitrary  linear mapping with the corresponding antisymmetric matrix $(\nabla_{ij})_{1\le i, j\le 8}.$ 
Let us show that $\nabla$ is a local derivation.

For any $i\in \{1, \ldots, 21\}$ denote by $D_{i}$ derivation of $M_{8}$ defined as in \eqref{derm8}
with the coefficients  $\alpha_{i}=1$ and $\alpha_{j}=0$ for all $j\neq i.$
In fact, 
\begin{align}\label{list}
\mathcal{B}=\left\{D_{i}: 1\le i\le 21\right\}.
\end{align}
Let $x=\sum\limits_{k=1}^8 x_{k}e_{k}\in M_{8}$ be a fixed non zero element. We need to find a derivation $D_{x}$ such that $\nabla(x)=D_{x}(x).$ 

Set $\nabla(x)=\sum\limits_{i=1}^8 y_ie_i.$ Note that 
\begin{align*}
\sum\limits_{i=1}^8 x_{i}y_{i} & =\sum\limits_{i=1}^8\sum\limits_{j\neq i} \nabla_{i j} x_{i}x_{j}=0,  \end{align*}
because $\nabla_{ij}=-\nabla_{ji}$ for all $1\le i, j\le 8.$
This means that 
\begin{align}\label{perp}
\nabla(x)\in x^\perp=\left\{z=\sum\limits_{i=1}^8 z_ie_i\in M_8: (x,z)=\sum\limits_{i=1}^8 x_iz_i=0\right\}.
\end{align}

Take a derivation $D_{e_1}$ of $M_8$ such that $\nabla(e_1)=D_{e_1}(e_1).$ If necessary replacing $\nabla$ with $\nabla-D_{e_1}$ we can assume that
$\nabla(e_1)=0.$ Then $\nabla_{1 i}=\nabla_{i1}=0$ for all $1\le i\le 8.$ Thus $\nabla$ maps $M_8$ into $e_1^\perp,$ that is, \begin{align}\label{nablax}
\nabla(x)\in e_1^\perp
\end{align}
 for all $x\in M_8.$

Let  us  consider the following possible two cases.

Case 1.  
Let  $x=x_{1}e_{1}.$ 
Then 
$$
\nabla(x)=x_1\nabla(e_1)=0=D_x(x),
$$
where $D_x$ is a trivial derivation.

Case 2. Let $x=\lambda_0 e_1+\lambda x_1,$ where $\lambda\neq 0$ and $x_1\in e_1^\perp$ with $x_1^2=-1.$

Since $\nabla(e_1)=0,$ it follows that 
$$
y=\nabla(x)=\nabla(\lambda_0 e_1+\lambda x_1)=\lambda\nabla(x_1).
$$
Combining \eqref{perp} and \eqref{nablax} we obtain that
\begin{align}\label{yperp}
y\in \{e_1, x\}^\perp.     
\end{align}
Thus $y$ represents as $y=\mu y_1,$ where $y_1^2=-1.$ 
By \eqref{yperp} we obtain that
\begin{align*}
y_1\in \{e_1, x_1\}^\perp.     
\end{align*}
By Proposition \ref{aut} there exists an automorphism $\Phi$ such that 
\begin{center}
$\Phi(e_2)=x_1$ and $\Phi(e_3)=y_1.$
 \end{center}
 Take a derivation $D=\frac{\mu }{\lambda}D_1,$ where $D_1$ is a derivation from the list \eqref{list}.
 Note that $D_1(e_1)=0$ and $D_1(e_2)=e_3.$
 Then the following mapping  
 $$
 D_x=\Phi\circ D\circ \Phi^{-1}
 $$
is a derivation. We have that 
\begin{align*}
D_x(x) & = \Phi\circ D\circ \Phi^{-1}(\lambda_0e_1+\lambda x_1)=\Phi(D(\Phi^{-1}(\lambda_0e_1+\lambda x_1)))=\Phi(D(\lambda_0e_1+\lambda e_2))\\
& =\Phi(D(\lambda  e_2))= \mu  \Phi\left(e_3\right)=\mu y_1=y=\nabla(x).
\end{align*}
The proof is completed.
\end{proof}


At the end we formulate Problem concerning  $2$-local derivations  of $M_8.$ 
Likewise as in the proof of Theorem \ref{2local}, we can obtain that any $2$-local derivation of $M_8$ is linear. In particular, it is a local derivation. In this regard, the following question arises.

\begin{problem}\label{2localM8}
Is any $2$-local derivation of   $M_{8}$  a derivation?
\end{problem}


\begin{thebibliography}{00}

 

\bibitem{aa17}
  Ayupov Sh.,    Arzikulov F.,
    $2$-Local derivations on associative and Jordan matrix rings over commutative rings,
    Linear Algebra and its Applications,  522 (2017), 28--50.


 
 
\bibitem{ak17}
Ayupov Sh.,  Kudaybergenov K.,
    $2$-Local automorphisms on finite-dimensional Lie algebras,
    Linear Algebra and its Applications, 507 (2016), 121--131.


\bibitem{ak16}
Ayupov Sh., Kudaybergenov K.,
    Local derivations on finite-dimensional Lie algebras,
    Linear Algebra and its Applications, 493 (2016), 381--398.



\bibitem{akoz}
Ayupov Sh.,  Kudaybergenov K.,  Omirov B.,
    Local and $2$-local derivations and automorphisms on simple Leibniz algebras,
    Bulletin of the Malaysian Mathematical Sciences Society, 43 (2020), 3, 2199--2234.


\bibitem{AyuKudRak16}
Ayupov Sh.,  Kudaybergenov K.,  Rakhimov I.,
    $2$-Local derivations on finite-dimensional Lie algebras,
    Linear Algebra and its Applications, 474 (2015), 1--11.

\bibitem{c19}
Costantini M.,
    Local automorphisms of finite dimensional simple Lie algebras,
    Linear Algebra and its Applications, 562 (2019), 123--134.

 
 \bibitem{filippov}
 Filippov V., 
$n$-Lie algebras, 
Siberian Mathematical Journal,  26 (1985), 6, 126--140.
 
 
 
\bibitem{k90}
Kadison R.,
    Local derivations,
    Journal of Algebra, 130 (1990), 2, 494--509.


\bibitem{Ivan}
Kaygorodov I., 
    $(n+1)$-ary derivations of semisimple Filippov algebras,    
    Mathematical Notes, 96 (2014), 1--2, 206--216.


\bibitem{k14}
 Kaygorodov I.,
    On $(n+1)$-ary derivations of simple $n$-ary Malcev algebras,
    St. Petersburg Mathematical Journal, 25 (2014),  4, 575--585.


 
\bibitem{Khrypchenko19}
Khrypchenko M.,
    Local derivations of finitary incidence algebras,
    Acta Mathematica Hungarica, 154 (2018), 1, 48--55.


\bibitem{Ling} 
Ling W., 
    On Structure of $n$-Lie Algebras, 
    Thesis (Universität Siegen, Siegen, 1993).

\bibitem{app} 
Pozhidaev A.,
    $n$-ary Malcev algebras,  
    Algebra Logic, 40 (2001), 3, 170--182 
 
\bibitem{paulo} 
 Pojidaev A.,  Saraiva P., 
    On derivations of the ternary Malcev algebra $M_8,$ 
    Communications in Algebra, 34 (2006), 10, 3593--3608.
 
\bibitem{Semrl97}
\v{S}emrl P.,
    Local automorphisms and derivations on $B(H)$,
    Proceedings of the American Mathematical Society, 125 (1997), 2677--2680.

\end{thebibliography}
\end{document}